\documentclass[reqno]{amsart}

\title[Koppelman formulas on flag manifolds]{Koppelman formulas on flag manifolds}

\author{H\aa kan Samuelsson \& Henrik Sepp\"{a}nen}
\thanks{The second author was supported by a postdoctoral fellowship from the 
Swedish Research Council.}

\date{\today}

\address{H. Samuelsson, Matematisk Institutt, Universitetet i Oslo, Postboks 1053 Blindern, 0316
Oslo, Norway}
\email{haakan.samuelsson@gmail.com}

\address{Henrik Sepp\"{a}nen,Universit\"{a}t Paderborn,
Fakult\"{a}t f\"{u}r Elektrotechnik, Informatik und Mathematik,
Institut f\"{u}r Mathematik, Warburger Str. 100
33098 Paderborn}
\email{henriksp@math.upb.de}

\usepackage[english]{babel}
\usepackage{amsmath}
\usepackage{amsthm,amssymb,latexsym}
\usepackage{t1enc}
\usepackage{mathrsfs}
\usepackage{graphicx}
\usepackage{xypic}

\newtheorem{proposition}{Proposition}
\newtheorem{theorem}[proposition]{Theorem}

\theoremstyle{definition}
\newtheorem{definition}[proposition]{Definition}

\newtheorem{remark}[proposition]{Remark}

\newcommand{\C}{\mathbb{C}}

\newcommand{\debar}{\bar{\partial}}

\newcommand{\Z}{\mathbb{Z}}

\newcommand{\z}{\zeta}

\newcommand{\Hom}{\textrm{Hom}}

\def\newop#1{\expandafter\def\csname #1\endcsname{\mathop{\rm #1}\nolimits}}
\newop{span}

\begin{document}
\nocite{*}
\bibliographystyle{plain}

\begin{abstract}
We construct Koppelman formulas on manifolds of flags in $\C^N$ for forms
with values in any holomorphic line bundle as well as in the tautological 
vector bundles and their duals. As an application we obtain new explicit proofs of some 
vanishing theorems of the Bott-Borel-Weil type
by solving the corresponding $\debar$-equation.
We also construct reproducing kernels for harmonic $(p,q)$-forms 
in the case of Grassmannians.
\end{abstract}

\maketitle
\thispagestyle{empty}

\section{Introduction}
The classical Koppelman formula is an integral formula
of the kind
\begin{equation}\label{huvudformel}
\varphi(z)=\int_{\partial \Omega} K\wedge \varphi+
\int_{\Omega} K\wedge \debar \varphi + \debar_z \int_{\Omega} K\wedge \varphi 
+ \int_{\Omega} P\wedge \varphi,
\end{equation}
which represents differential forms on a smoothly
bounded domain $\Omega \subseteq \C^n$.
Here $K$ is an integrable differential form on $X \times X$, and
$P$ is a smooth one.
The formula \eqref{huvudformel} is equivalent to the
equation of currents
\begin{equation}
\debar K=[\Delta]-P \label{huvudekv},
\end{equation}
where $[\Delta]$ is the $(n,n)$-current given by integration
over the diagonal $\Delta \subseteq \C^n \times \C^n$.

In this paper we consider the case of the complex manifold $X$ of all
flags $V_1 \subseteq \cdots \subseteq  V_k$ of linear subspaces of 
$\C^N$ of a given type, and construct integral formulas for smooth 
differential forms with values in certain holomorphic vector bundles
$V \rightarrow X$.
To be more precise we construct integral kernels $K$ and $P$,
integrable and smooth respectively, so that 
the formula \eqref{huvudekv}
holds for any smoothly bounded domain $\Omega\subseteq X$ and
smooth $(p,q)$-form 
on $\overline{\Omega}$ with values in the holomorphic line bundle
$\mathcal{L} \rightarrow X$ (or in any of the tautological vector bundles).
Here $K$ is an integrable section of the bundle 
$V \boxtimes V^* \otimes \wedge T^*_{\C} 
\rightarrow X \times X$, and $P$
is a smooth section of the same bundle. In this setting
the current $[\Delta]$ is given by the tensor
product of the integration current associated to 
$\Delta \subseteq X \times X$ and the identity 
section of $(V \boxtimes V^*)|_{\Delta}$.

We apply the formula \eqref{huvudformel} to prove
certain vanishing theorems for Dolbeault cohomology
groups with values in line bundles. 
Indeed, if we let $\Omega$ to be equal to $X$, the
boundary term on the right hand side of the formula
vanishes. Therefore $\debar$-closed forms $\varphi$
are represented by the third plus the fourth term, 
and hence the third terms provides a solution to
the $\debar$-equation if one can prove that the
fourth term vanishes.

The approach of this paper is a straightforward generalization 
of the one in \cite{GSS} which deals with Grassmannians. That paper in 
turn is an application of (a slight generalization of) the
method developed in \cite{elin} for constructing integral
formulas on complex manifold. The key ingredient is to 
have a manifold $X$ admitting the \emph{Diagonal Property}, 
i.e., that $X \times X$ admits a holomorphic vector bundle
$E \rightarrow X \times X$ of rank equal to the dimension of
$X$, and a holomorphic section $\eta$ of $E$ which
vanishes to the first order on the diagonal 
$\Delta \subseteq X \times X$ and is nonzero elsewhere. 
We will therefore be very brief and refer to \cite{GSS}
and \cite{elin} for the framework in which our contructions
take place, as well as for a historical background on 
integral formulas.

The paper is organized as follows. In Section 2 we recall some
preliminaries on flag manifolds and related vector bundles. 
In Section 3 we show that $X$ admits the Diagonal property, 
and use this to give a Koppelman formula for scalar valued
differential forms. In Section 4 we construct weights which allow
us to give integral formulas for forms with values in vector 
bundles. Section 5 is concerned with applications. We
prove vanishing theorems for Dolbeault cohomology 
groups, and, in the special case of Grassmannians, we also 
identify the fourth term on the right hand side in 
\eqref{huvudformel} with the reproducing kernel for the space 
of harmonic $(p,q)$-forms.

{\bf Acknowledgement:} We would like to thank Professor P. Pragacz for 
having sent us a preprint version of his paper \cite{pragacz} on the Diagonal 
Property.

\section{Preliminaries}

\subsection{Flag manifolds}
A flag $\mathbf{F}$ of type $\underline{d}:=(d_1,\ldots,d_k)$ is a $k$-tuple 
$\mathbf{F}=(V_1,\ldots,V_k)$ of subspaces $V_1 \subset V_2 \subset \ldots \subset V_k$ of 
$\C^N$, where  $\mbox{dim} V_i=d_i$, and $V_k=\C^N$.
We fix the flag $\mathbf{F}_0=(V^0_1,\ldots, V^0_k)$, with 
$V^0_i=\mbox{Span}_{\C}\{e_1, \ldots, e_{d_i}\}$, 
where $\{e_1, \ldots, e_N\}$ is the standard basis for $\C^N$.

Let $X$ be the set of all flags of subspaces of type $(d_1,\ldots,d_k)$ of $\C^N$.
The group $G:=GL(N,\C)$ then acts transitively on $X$, and the stabilizer of the 
reference point $(V^0_1,\ldots, V^0_k)$ is 
\begin{equation*}
P:=\left\{ \left(\begin{array}{ccc}
A_1 & * & *\\
0 & \ddots & *\\
0 & \cdots & A_k
\end{array}\right)
\in GL(N,\C)
\right\}
\end{equation*}
where $A_1$ is a block of size $d_1 \times d_1$, and $A_i$ is a block
of size $(d_i-d_{i-1}) \times (d_i-d_{i-1})$ for $i=2,\ldots,k$.

On the level of Lie algebras, we have the decomposition
\begin{eqnarray*}
\mathfrak{g}=\mathfrak{n} \oplus \mathfrak{p},
\end{eqnarray*}
with 
\begin{eqnarray*}
\mathfrak{p}&:=&\left\{ \left(\begin{array}{ccc}
A_1 & * & *\\
0 & \ddots & *\\
0 & \cdots & A_k
\end{array}\right)
\in M_{NN}(\C)
\right\},\\
\mathfrak{n}&:=&
\left\{ \left(\begin{array}{cccc}
0 & 0 & \cdots & 0\\
B_{11} & 0 & \cdots & 0\\
\vdots & \ddots & \cdots &0\\
B_{(k-1)\,1} & \cdots & B_{(k-1)\,(k-1)} & 0
\end{array}\right)
\in M_{NN}(\C)
\right\},
\end{eqnarray*}
where $B_{i1}$ is a block of size $(d_{i+1}-d_i) \times d_1$, and 
$B_{ij}$ is a block of size $(d_{i+1}-d_i) \times (d_j-d_{j-1})$,
$j=2,\ldots,k-1$.
From this, we see that $X$ is a complex manifold of dimension
\begin{equation*}
n:=\mbox{dim}_{\C} \mathfrak{n}=
\sum_{j=1}^{k-1}(N-d_j)(d_j-d_{j-1})=\sum_{j=1}^{k-1} d_j(d_{j+1}-d_j).
\end{equation*}
Here we use the convention that $d_0=0$.
\subsection{Vector bundles over $X$}
For each $i=1,\ldots,k$, we have a tautological holomorphic vector bundle
\begin{equation*}
H_i:=\{(\mathbf{F},v) \in X \times \C^N |\, v \in V_i \},
\end{equation*}
and a corresponding quotient bundle $F_i:= X \times \C^N/H_i$.
We let $q_i: X \times \C^N \rightarrow F_i$ denote the holomorphic projection map, 
and 
\begin{equation}\label{storkvot}
Q_i: F_i \rightarrow F_{i+1}, \quad v + H_i \mapsto v + H_{i+1}
\end{equation}
be the natural quotient map.

Note that the restriction of the Euclidean metric on $\C^N$ equips each 
$H_i$ with a Hermitian metric which is invariant under the action of
the group $U(N)$. Moreover, we can indentify the quotient bundle $F_i$ 
with the orthogonal complement, $H_i^{\perp}$ of $H_i$.
We let $\varphi_i:F_i \rightarrow H_i^{\perp}$ denote this identification, the
inverse of which is the restriction of $q_i$ to $H_i^{\perp}$.
These identifications allow us to define $U(N)$-invariant Hermitian metrics on the
$F_i$.
In the sequel, we will be concerned with the Cartesian product $X \times X$.
We will denote points in the first component by $z$, and points in the second
component by $\zeta$. We let $H_{i,z} \rightarrow X \times X$ denote the 
the pullback bundle of $H_i$ with respect to the projection of $X \times X$ onto
the first component. Analogously, we define $H_{i, \zeta}$ and the correponding
constructions for the $F_i$. Finally, we let $q_{i,z}: X \times X \times \C^N \rightarrow
F_{i,z}$ denote the natural quotient morphism.

\subsection{The Picard group of $X$}
Since $X$ can be described as quotient of $G$ by a parabolic subgroup, it
follows that any holomorphic line bundle over $X$ is homogeneous
under $G$, i.e., corresponds to a holomorphic character $\chi:P \rightarrow \C^{\times}$
(cf. \cite{snow}).
It is well known that the group of holomorphic characters of $P$ 
is generated by the characters $\chi_1, \ldots, \chi_k$, where
\begin{equation*}
\chi_i(P)=\det A_1 \cdots \det A_i,
\end{equation*}
where $P=LU$ is the Levi decomposition of $P$ with 
\begin{eqnarray*}
L&:=&\left\{ \left(\begin{array}{ccc}
A_1 & 0 & 0\\
0 & \ddots & 0\\
0 & \cdots & A_k
\end{array}\right)
\in GL(N,\C)
\right\},\\
U&:=&\left\{ \left(\begin{array}{ccc}
I_{d_1} & * & *\\
0 & \ddots & *\\
0 & \cdots & I_{(d_k-d_{k-1})\,(d_k-d_{k-1})}
\end{array}\right)
\in GL(N,\C)
\right\}.
\end{eqnarray*}

The corresponding homogeneous line bundle 
\begin{equation*}
L_i:=G \times_{\chi_i} \C \label{generator}
\end{equation*}
is equivalent to $\det H_i:=\wedge^{d_i}H_i$ by
the morphism of $G$-homogeneous line bundles given by
\begin{equation*}
[(g,v)] \mapsto g.v,
\end{equation*}
where the $G$-action on the right hand side is the linear action on $\wedge^{d_i} \C^N$
induced from the standard action on $\C^N$.
The dual line bundles $L_i^{-1}:=L_i^*$ are ample, and we have isomorphisms
$\wedge^{d_i} \C^N \rightarrow H^0(X,L_i^{-1})$ given by
\begin{eqnarray*}
v &\mapsto & s_v,\\
s_v(z).u&:=&\langle u,v \rangle, \quad v \in \wedge^{d_i} \C^N, 
u \in \wedge^{d_i} H_{i,z},
\end{eqnarray*}
where the pairing $\langle \cdot, \cdot \rangle$ is given by the natural inner 
product on $\wedge^{d_i} \C^N$ induced from the one on $\C^N$.

\section{The Diagonal Property}
Consider the holomorphic vector bundle $\sum_{i=1}^{k-1} F_{i,z} \otimes H_{i,\zeta}^*$, 
which has rank $\sum_1^{k-1}d_j(N-d_j)$. We define a holomorphic section $\eta$ of this bundle by
\begin{eqnarray*}
\eta(z,\z)&:=&\sum_{i=1}^{k-1}\eta_i(z,\zeta), \\
\eta_i(z,\zeta)\, v&:=&q_{i,z}(v), \quad v\in H_{i,\zeta}\subset \C^N. \nonumber
\end{eqnarray*}
Then $\eta$ defines the diagonal $\Delta \subset X\times X$, i.e., $\eta$ vanishes to first order along
$\Delta$ and is non-zero elsewhere. Moreover, $\eta$ is a section of the subbundle $E:=\ker \Phi$, where
\begin{eqnarray*}
\Phi: \sum_{i=1}^{k-1}\Hom((H_i)_{\z},(F_i)_z)
&\rightarrow & \sum_{i=1}^{k-2}\mbox{Hom}((H_i)_{\z},(F_{i+1})_z), \\ \nonumber
\Phi(\sum_{i=1}^{k-1} f_i)&:=&\sum_{i=1}^{k-2} f_{i+1}|_{(H_i)_{\zeta}}-Q_{i} \circ f_i.
\end{eqnarray*}

The map $\Phi$ is surjective, and it follows that 
$E$ is of rank $n$ ($=\mbox{dim}\, X$). We thus have
\begin{proposition}
The section $\eta$ is a holomorphic section of the rank-$n$ bundle $E$ and 
$\eta$ defines the diagonal $\Delta \subset X \times X$.
\end{proposition}

\begin{remark}
The construction of $\eta$ occurred to us rather soon as a generalization
of the construction for Grassmannians (cf. \cite{GSS} ). However, we did not
realize that it defines a section of a subbundle of the right rank, and
therefore dismissed it. It was pointed out to us by Prof. Pragacz that it
actually is a section of $E$. We are very grateful to him for this.
\end{remark}

Equip $E$ with the $U(N)$-invariant Hermitian metric (here $U(N)$ is identified with 
the diagonal subgroup of $U(N) \times U(N)$) 
$\langle \cdot, \cdot \rangle_E$ induced from the 
Euclidean metric on $\C^N$ via the tautological vector bundles $H_i$ and
their quotients $F_i$. Let $D_E$ be the Chern connection 
with respect to this metric, and let $\Theta_E$ be
its curvature, and let $\widetilde{D_E}$ and $\widetilde{\Theta_E}$
denote their natural images as sections of the vector bundle
\begin{equation*}
G_E:=\Lambda [T^{\ast}(X \times X) \oplus E \oplus
E^{\ast}] \to X \times X. \label{G_E}
\end{equation*}

Let $\delta_\eta$ be mapping $\wedge E^* \rightarrow \wedge E^*$ 
defined by interior multiplication by $\eta$, and let $\sigma$ be the 
section of $E^*$ given by forming the inner product
with the normalization of $\eta$, i.e.,
\begin{equation*}
\sigma(z,\zeta)(v):=\frac{\langle v,\eta(z,\zeta)\rangle_E}{|\eta(z)|_E^2}.
\end{equation*}
Then $\delta_\eta \sigma=1$ on $X\times X\setminus\Delta$. Moreover, put
$u=\frac{\sigma}{\nabla_\eta \sigma}=
\sum_{k=1}^n \sigma \wedge (\debar \sigma)^{k-1}$, which is a
section of $G_E$.
Finally, let
$\int_E \colon G_E \rightarrow T^*(X \times X)$ be the mapping which 
takes a point to the coefficient of the 
term with degree $n$ in both $E$ and $E^*$.
We then have the following Koppelman formula for scalar
valued differential forms, see Theorem 3.4 in \cite{elin}.
\begin{theorem}
Define $K$ and $P$ by
\begin{equation*}
K(z,\zeta):=\int_E u\wedge \left(\frac{\widetilde{D_E\eta}}{2\pi i}
+\frac{i \widetilde{\Theta}_E}{2\pi}\right)_n, \quad 
P(z,\zeta):=\int_E \left(\frac{\widetilde{D_E\eta}}{2\pi i}
+\frac{i \widetilde{\Theta}_E}{2\pi}\right)_n,
\end{equation*}
where $(\cdot)_n=(\cdot)^n/n!$.
Then $K$ is integrable and $P$ is smooth, and moreover, $K$ and $P$ satisfy the equation of currents
\begin{equation*}
\debar K=[\Delta]-P.
\end{equation*}
\end{theorem}

Notice that for degree reasons we have 
$P=\int_E (i \widetilde{\Theta}_E/2\pi)_n=\det (i\Theta/2\pi)=c_n(E)$
so that $P$ is the $n$th Chern form of $E$.

\section{Weighted integral formulas}
We recall the definition of a weight for a holomorphic vector bundle $V\to X$.
We first put $G_{E,V}:= \Hom(V_{\zeta},V_z)\otimes G_E$, and
define the operator
\begin{equation*}
\nabla_\eta:=\delta_\eta-\debar
\end{equation*}
on the space of smooth sections of $G_{E,V}$.

\begin{definition}
Let $g=g_{0,0}+\ldots + g_{n,n}$ be a smooth section of $G_{E,V}$, where $g_{k,k}$ is a section of 
$\Lambda^k E^* \wedge T^*_{0,k}(X\times X)$. Then $g$ is called a {\em weight} for $V$
if $\nabla_{\eta}g=0$ and 
$g_{0,0}\mid_{\Delta}=\rm{Id}\in \Hom(V_{\zeta},V_z)\mid_{\Delta}$.
\end{definition}

Now, following \cite{elin}, p.\ 54, one easily shows that
\begin{equation}\label{KoP}
K_g(z,\zeta):=\int_E g \wedge u \wedge \left(\frac{\widetilde{D_E\eta}}{2\pi i}+\frac{i \tilde{\Theta}_E}{2\pi}\right)_n,\,\,
P_g(z,\zeta):=\int_E g \wedge \left(\frac{\widetilde{D_E\eta}}{2\pi i}+\frac{i \tilde{\Theta}_E}{2\pi}\right)_n
\end{equation}
satisfy \eqref{huvudekv}. 

\subsection{Weights for the tautological vector bundles}
To begin with, we define a section $\gamma_{0,i}$ of the bundle $H_i \boxtimes H_i^*$
by
\begin{equation*}
\gamma_{0,i}(z,\zeta)v:=\pi_{H_{i,z}}v, \quad v \in H_{i,\z}.
\end{equation*}
Next, we define $\gamma_{1,i} \in 
\Gamma(X \times X, (H_i \boxtimes H_i^* ) \otimes E^* \otimes T^*_{0,1}$. Let 
$\xi$ and $v$ be (germs of) smooth sections of $E$ and $H_{i,\z}$ respectively, 
and let $\xi$ be decomposed as $\xi=\xi_1+\cdots+\xi_k$, where 
$\xi_i \in (F_i)_z \otimes (H_i^*)_{\zeta}$. 
Then we let
\begin{eqnarray*}
\gamma_{1,i}(\xi \otimes v):=-\pi_{H_{i,z}}(\debar\varphi_i(\xi_i(v))).
\end{eqnarray*}
As a linear operator $\Gamma(X \times X, E \otimes H_{i,\z}) \rightarrow
\Gamma(X \times X, H_{i,z}\otimes T^*_{0,1})$ is $C^{\infty}(X \times X)$-linear
(cf. \cite{GSS}), and hence defines a section of the appropriate bundle.
\begin{proposition}
The section $G_i=G_i(z,\zeta):=\gamma_{0,i}+\gamma_{1,i}$ is a weight for the
tautological vector bundle $H_i$.
\end{proposition}

\begin{proof}
In view of the identity $\delta_{\eta}\gamma_{1,i}=\delta_{\eta_i}\gamma_{1,i}$, the
statement follows immediately from the proof of the analogous statement for the
Grassmannian of $d_i$-dimensional subspaces of $\C^N$ (cf. \cite{GSS}).
\end{proof}

\subsection{Weights for line bundles}
We recall from \cite{GSS} that we can form exterior powers, tensor powers, and
duals  of weights. The following proposition is an immediate consequence
of this construction.
\begin{proposition}
Define the section $g_i:=\wedge^{d_i}G_i$. Then we have
\begin{itemize}
\item[a)] $g_i$ is a weight for the line bundle $L_i$,
\item[b)] the section $g_i^{-1}:=g_i^*$ is a weight for $L_i^*$,
\item[c)] the section $g_1^{m_1} \otimes g_k^{m_k}$ 
is a weight for the line bundle $L$ with factorization 
$L=L_1^{m_1} \otimes \cdots \otimes L^{m_k}$, 
into integer powers of the generators $L_i$.
\end{itemize}
\end{proposition}
\section{Applications}

\subsection{Vanishing theorems}
In this section we prove a vanishing theorem of the Bott-Borel-Weil type 
for Dolbeault cohomology. We reduce the problem to the case of Grassmannians
by fiber bundle techniques. In the Grassmannian case the corresponding 
vanishing theorems are proved explicitly by solving
the $\debar$-equation using Koppelman formulas. 

\begin{theorem} \label{manzanita}
For $i=1,\ldots, k$,  the cohomology group $H^{q}(X, L_i^r)$ is trivial for 
$r \leq 0$ and $q>0$.





\end{theorem}

\begin{proof}
Let $Y:=\mbox{Gr}(d_i,N)$ be the Grassmannian of $d_i$-dimensional subspaces 
of $\C^n$, and let $\widetilde{L_i} \rightarrow Y$ be the determinant of 
the tautological vector bundle over $Y$. 
Let $\tau:X \rightarrow Y$ be the natural projection $(V_1,\ldots, V_k) \mapsto V_i$.
Then $\tau$ defines a holomorphic fiber bundle with fiber 
\begin{equation}
\tau^{-1}(W)=X_W^1 \times X_W^2, \label{E: fiber}
\end{equation}
where 
$X^1_W$ is the manifold of all flags $V_1 \subset \ldots \subset V_{i-1} \subset W$ 
with $\mbox{dim} V_j=d_j$, for $j=1,\ldots, i-1$, and $X^2_W$ is the manifold of all 
flags $W \subset V_{i+1} \subset \ldots \subset V_k=\C^N$, with 
$\mbox{dim} V_j=d_j$, $j=i+1,\ldots ,k$. Moreover, $\tau^*\widetilde{L_i}=L_i$.
Let $\{U_j\}$ be an open cover of $Y$ such that both the bundles 
$\widetilde{L_i} \rightarrow Y$ and $X \rightarrow Y$ are locally 
trivial over each $U_j$, and such that $U_j$ is biholomorphically equivalent 
to an open ball in $\C^m$ (where $m$=dim $Y$).
Since $\tau$ has compact and connected fibers, the identity 
\begin{equation}
H^0(\tau^{-1}(O),L_i^r) \cong H^0(O, \widetilde{L_i}^r) \label{E: localiso}
\end{equation}
holds for any sufficiently small open ball $O \subseteq Y$, and $r \in \Z$.
Let $\mathcal{F}(r)$ denote the sheaf of germs of holomorphic sections of
$L_i^r$, and let $\widetilde{\mathcal{F}}(r)$ denote the sheaf of germs of holomorphic 
sections of $\widetilde{L_i}^r$. 
From the isomorphism \eqref{E: localiso} it follows that we have the
isomorphism 
\begin{equation*}
\tau_*\mathcal{F}(r) \cong \widetilde{\mathcal{F}}(r), 
\end{equation*}
of sheaves over $Y$.
Here $\tau_*$ is the direct image of $\mathcal{F}(r)$.
Recall that that the $q$th direct image of $\mathcal{F}(r)$, $R^q\tau_*\mathcal{F}(r)$,  
is the sheafification of the presheaf $O \mapsto H^q(\tau^{-1}(O),\mathcal{F}(r))$ 
over $Y$. 
For a sufficiently small open ball $O$, {\it e.g.}, for $O \subseteq U_j$ for some 
$j$, we have
\begin{equation}
H^q(\tau^{-1}(O),\mathcal{F}(r)) \cong H^q(O \times X_W^1 \times X_W^2, \mathcal{O}),
\label{E: cohfiber}
\end{equation}
where the right hand side denotes the cohomology of the holomorphic structure
sheaf, i.e., the cohomology of the trivial line bundle. 
We can now prove by induction on $n=\mbox{dim}X$ that $H^q(X,L_i^r)=0$ for 
$r \leq 0$. 
If $n=1$, then $X=\mbox{P}^1(\C)$, and the statement is proved by 
explicitly by solving the $\debar$-equation using a Koppelman formula
in \cite[Thm. 12]{GSS}. 
For the induction step, assume that the claim is proved for 
all manifolds of flags of dimension strictly smaller than 
$n$. In particular, $H^q(X_W^1,\mathcal{O})=0$, and $H^q(X_W^2,\mathcal{O})=0$ for 
$q>0$ (cf. \eqref{E: fiber}). By the K\" unneth formula, we now have
\begin{equation*}
H^q(X_W^1 \times X_W^2,\mathcal{O})=
\bigoplus_{q_1+q_2=q}H^{q_1}(X_W^1,\mathcal{O}) 
\otimes H^{q_2}(X_W^2,\mathcal{O})=0
\end{equation*} 
if $q>0$.
Since $O$ is Stein we then also have 
$H^q(O \times X_W^1 \times X_W^2, \mathcal{O})=0$ for $q>0$ 
(cf. \cite[Ch. IX, Cor. 5.22]{Demailly}), so that the identity 
\eqref{E: cohfiber} yields that $R^q\tau_*\mathcal{F}(r)=0$ for any $r \in \Z$. 
From the Leray spectral sequence it thus follows that for $q>0$ we have the 
isomorphism
\begin{equation*}
H^q(X,L_i^r) \cong H^q(Y,\widetilde{L_i}^r)
\end{equation*}
(cf. \cite[p.234, Cor. 13.9]{Demailly}).
In \cite[Thm. 12]{GSS}, the claim that $H^q(Y,\widetilde{L_i}^r)=0$ 
for $q>0$ and $r \leq 0$ is proved by solving the $\debar$-equation 
explicitly using a Koppelman formula. 
The statement therefore follows from the induction hypothesis. 

\end{proof}

\subsection{Harmonic forms}

In this section we will construct reproducing kernels for harmonic forms in 
the special case of Grassmannians. Let therefore $X:=G(m,N)$ be the Grassmannian
of $m$-dimensional subspaces of $\C^N$, and let $A^{p,q}(X)$ denote 
the space of $(p,q)$-forms on $X$.

We recall the Hodge-decomposition
\begin{eqnarray}
A^{p,q}(X)=\mathcal{H}^{p,q}(X) \oplus 
\debar A^{p,q-1}(X) \oplus
\debar ^*A^{p,q+1}(X) \label{hodge}
\end{eqnarray}
into an orthogonal sum with respect to the $L^2$-inner product 
\begin{equation*}
\langle \varphi, \psi \rangle:= \int_X \varphi \wedge *\psi, 
\end{equation*}
where $*$ denotes the (antilinear) Hodge operator defined by a 
$U(N)$-invariant Hermitian metric on $T(X)$. Here $\mathcal{H}^{p,q}(X)$
denotes the harmonic $(p,q)$-forms, i.e., the kernel of the
operator 
$\debar\debar^*+\debar^*\debar: A^{p,g}(X) \rightarrow A^{p,q}(X)$.

\begin{theorem}
The operator
$A^{p,q}(X) \rightarrow \mathcal{H}^{p,q}(X), \,\,
\varphi \mapsto \int_{\zeta} \varphi\wedge P$
is the orthogonal projection onto the space 
of harmonic $(p,q)$-forms.
\end{theorem}

\begin{proof}
To begin with, observe that the curvature tensor $\Theta$ of 
$E=F_z \otimes H_{\zeta}^*$ is associated to $U(N)$-invariant metrics 
on the respective bundles, and is therefore
$U(N) \times U(N)$-invariant. 
Since $X$ is a Riemannian symmetric space, it follows, e.g., from
\cite[Exercise B2, p.227]{helg}, that a $U(N)$-invariant $(p,q)$-form on $X$ is 
harmonic with respect to the Laplacian $dd^*+d^*d$. But since $X$ is K\"{a}hler we have
\begin{equation*}
\debar \debar^*+\debar^*\debar=\frac{1}{2}(dd^*+d^*d)
\end{equation*}
and we see that $U(N)$-invariant $(p,q)$-forms indeed are harmonic.
It follows that $z\mapsto P(z,\zeta)$, $\zeta\mapsto P(z,\zeta)$, and $\int\varphi\wedge P$
are harmonic for any $\varphi\in A^{p,q}(X)$.

Secondly, if $\varphi$ is harmonic, the Koppelman formula yields
\begin{equation*}
\varphi=\debar \int \varphi \wedge K+\int \varphi \wedge P.
\end{equation*}
Since the second term on the right hand side is harmonic, it
follows from \eqref{hodge} that $\debar \int \varphi \wedge K=0$, i.e., 
thet $\varphi=\int  \varphi \wedge P$.

Assume now that $\varphi=\debar \psi$ for some $\psi \in A^{p,q-1}(X)$. 
Then, since $\zeta\mapsto P(z,\zeta)$ is $\debar$-closed,
\begin{eqnarray*}
\int \debar \psi\wedge P=\int \debar (\psi \wedge P)=\int d(\psi \wedge P)=0.
\end{eqnarray*}

Finally, if $\varphi=\debar^* \psi$ for some $\psi \in A^{p,q+1}(X)$, 
\begin{eqnarray*}
\int \debar^*\varphi \wedge P=\pm \langle \debar^*\psi, *P \rangle
=\pm \langle \psi, \debar * P \rangle
=\pm \langle  \psi, *\debar^* P\rangle
=0
\end{eqnarray*}
since $\zeta\mapsto P(z,\zeta)$ is $\debar^*$-closed.
This proves the theorem.
\end{proof}


\begin{thebibliography}{99}


    





\bibitem{Demailly} \textsc{J.-P. Demailly:} Complex Analytic and Differential Geometry. 
Online book, available at: http://www-fourier.ujf-grenoble.fr/$\sim$demailly/.

\bibitem{elin} \textsc{E. G\"{o}tmark:} Weighted integral formulas on manifolds. 
\textit{Ark. Mat.}{\bf 46} (2008), 1, 43--68.

\bibitem{GSS}
\textsc{E. G\"{o}tmark, H. Samuelsson, and H. Sepp\"anen :} Koppelman formulas
on Grassmannians. 
\textit{J. Reine Angew. Math.} {\bf 640}  (2010), 101--115.

\bibitem{helg}
\textsc{S. Helgason:}
Differential geometry, Lie groups, and symmetric spaces,
\textit{Academic Press Inc.}, New York, 1978.







\bibitem{pragacz} 
\textsc{P. Pragacz:} Miscellany on the zero schemes of
sections of vector bundles.
\textit{Algebraic cycles, sheaves, shtukas,
and moduli, Trends in Mathematics}. Birkh\"auser (2007), 105--116.
    	





     


\bibitem{snow}
\textsc{D. M. Snow:} Homogeneous vector bundles.
\textit{Group actions and invariant theory (Montreal, PQ, 1988)},
CMS Conf. Proc., \textbf{10}, 193--205, Amer. Math. Soc.,
Providence, RI, 1989.





\end{thebibliography}
\end{document}